\documentclass[11pt]{amsart}

\usepackage{amsmath}
\usepackage{amssymb}
\usepackage{amsthm}
\usepackage[all,cmtip]{xy}
\usepackage{color}

\theoremstyle{plain}
\newtheorem{thm}{Theorem}[section]
\newtheorem{prop}[thm]{Proposition}
\newtheorem{lem}[thm]{Lemma}

\theoremstyle{definition}
\newtheorem{defn}[thm]{Definition}
\newtheorem{eg}[thm]{Example}

\theoremstyle{remark}
\newtheorem{rem}[thm]{Remark}
\newtheorem{claim}[thm]{Claim}

\newcommand{\bC}{\ensuremath{\mathbb{C}}}

\newcommand{\bP}{\ensuremath{\mathbb{P}}}

\newcommand{\bZ}{\ensuremath{\mathbb{Z}}}

\newcommand{\cL}{\ensuremath{\mathcal{L}}}

\newcommand{\cN}{\ensuremath{\mathcal{N}}}
\newcommand{\cO}{\ensuremath{\mathcal{O}}}

\newcommand{\cX}{\ensuremath{\mathcal{X}}}

\DeclareMathOperator{\Pic}{Pic}
\DeclareMathOperator{\Sing}{Sing}

\DeclareMathOperator{\Ker}{Ker}
\DeclareMathOperator{\Image}{Im}

\DeclareMathOperator{\id}{id}
\DeclareMathOperator{\pr}{pr}

\DeclareMathOperator{\rk}{rk}

\begin{document}

\title
[Non-K\"{a}hler Calabi-Yau manifolds]
{Examples of non-K\"{a}hler Calabi-Yau manifolds with arbitrarily large $b_2$}

\author{Taro Sano}
\address{Department of Mathematics, Faculty of Science, Kobe University, Kobe, 657-8501, Japan}
\email{tarosano@math.kobe-u.ac.jp}

\maketitle

\begin{abstract} 
	We construct non-K\"{a}hler Calabi-Yau manifolds of dimension $\ge 4$ with arbitrarily large 2nd Betti numbers by smoothing normal crossing varieties. 
	The examples have K3 fibrations over smooth projective varieties and their algebraic dimensions are of codimension $2$.

		\end{abstract}

\tableofcontents

\section{Introduction} 
In this paper, we say that a compact complex manifold $X$ is a {\it Calabi-Yau manifold} if its canonical bundle $\omega_X \simeq \cO_X$ and 
$H^i(X, \cO_X) = H^0(X, \Omega_X^i)=0$ for $0<i < \dim X$. Moreover, we say that $X$ is a {\it Calabi-Yau $n$-fold} if its dimension is $n$.  

Projective Calabi-Yau manifolds form an important class of algebraic varieties. 
Non-K\"{a}hler Calabi-Yau manifolds are well investigated in complex differential geometry including those without the condition on the cohomology groups (cf. \cite{MR2891478}, \cite{MR3372471}). 
Reid's fantasy \cite{MR909231} suggests the possible importance of non-K\"{a}hler Calabi-Yau manifolds in the classification of projective ones. 

One of the important problems on projective Calabi-Yau manifolds is whether their topological types are finite or not. 
Inspired by these background, we construct infinitely many non-K\"{a}hler Calabi-Yau manifolds with the following properties. 

\begin{thm}\label{thm:main}
For positive integers $m$ and $N \ge 4$, there exists a simply connected non-K\"{a}hler Calabi-Yau $N$-fold $X=X(m)$ such that 
\[
b_2(X) = \begin{cases}m+10 & N=4 \\ 
m+2 & N \ge 5 
\end{cases}, \ \ a(X)=N-2, 
\]
where $b_2(X)$ is the 2nd Betti number 
 and $a(X)$ is the algebraic dimension of $X$. The topological Euler number $e(X)$ can also be computed (Proposition \ref{prop:pi1Euler}). 
\end{thm}

Since the 2nd Betti number of $X$ can be arbitrarily large, the examples give infinitely many topological types of non-K\"{a}hler Calabi-Yau manifolds of dimension $N \ge 4$. 
As far as we know, these are the first examples of Calabi-Yau manifolds with infinitely many topological types in 
a fixed dimension $N \ge 4$ in our sense. 
%This answers the question \cite{MOF19} on MathOverflow

The examples suggest that  
there are very plenty of non-K\"{a}hler Calabi-Yau manifolds which are close to projective ones even in dimension $>3$. 
Note that direct products of known lower-dimensional Calabi-Yau manifolds  have non-zero holomorphic forms and so are not  themselves Calabi-Yau.

Calabi-Yau 3-folds with arbitrarily large $b_3$ and $b_2 =0$ are constructed by Clemens and Friedman (cf.\cite{MR1141199}, \cite{MR1410077}). 
Hashimoto and the author constructed non-K\"{a}hler Calabi-Yau 3-folds with arbitrarily large $b_2$ by smoothing normal crossing varieties in \cite{Hashimoto:aa}. 
 The idea of constructing projective Calabi-Yau manifolds by smoothing SNC varieties goes back to the papers \cite{MR1296351} and \cite{MR2658406}. In this paper, we construct examples by the same method as  \cite{Hashimoto:aa}.

%It would hopefully  be useful for understanding the geography of non-K\"{a}hler Calabi-Yau manifolds. 

\subsection{Comments on the construction} 
 The idea of the construction in \cite{Hashimoto:aa} was to consider distinct SNC varieties by gluing smooth varieties along their anticanonical divisors through an automorphism of infinite order. The point is that we use smooth varieties which are blow-ups of other varieties and the number of blow-up centers increases when we change the gluing isomorphisms. 
 
In this paper, we consider an SNC variety of the form $X_0 = X_1 \cup X_2$, where $X_1=\bP^2 \times T$ and $X_2$ is a blow-up of $X_1$ and $T \subset \bP^1 \times \bP^n$ is a hypersurface of bi-degree $(1, n+1)$.  
The essential ingredient in this paper is the Schoen type Calabi-Yau manifold with infinite automorphisms (cf. \cite{MOF2020}, \cite{MR923487}, \cite{MR1672045}). 
When $n=2$, the intersection of two irreducible components is the Schoen Calabi-Yau 3-fold which is a fiber product of two rational elliptic surfaces over $\bP^{1}$. 
We actually use isomorphisms of such Calabi-Yau ($N-1$)-folds which come from quadratic transformations on rational elliptic surfaces. 
We effectively use the description of general rational elliptic surfaces 
as hypersurfaces of $\bP^2 \times \bP^1$ of bidegree $(3,1)$. 

\begin{rem}
In  \cite{Hashimoto:aa}, examples of SNC Calabi-Yau 3-folds are constructed by using automorphisms of $(2,2,2)$-hypersurface $X_{(2,2,2)} \subset \bP^1 \times \bP^1 \times \bP^1$ induced by the covering involutions of double covers $ X \rightarrow \bP^1 \times \bP^1$. 
Note that, for $X=X_{(2, \ldots, 2)} \subset (\bP^1)^n$, the covering involutions 
of projections $X \rightarrow (\bP^1)^{n-1}$ are birational maps with indeterminacies when $n \ge 4$, 
thus we need new gluing isomorphisms to construct examples in higher dimensions. 
It is also not clear that the construction of Clemens--Friedman can be generalized to higher dimensions. 
 \end{rem}	

After finishing the manuscript, the author received an e-mail from Nam-Hoon Lee and he constructed a non-K\"{a}hler Calabi-Yau 4-fold by smoothing SNC varieties \cite{Lee:2021aa}. 	
	
 \section{Preliminaries}	

We can construct an SNC variety by gluing two smooth proper varieties along their smooth divisors as follows. 

\begin{prop}\label{prop:gluingSNC}
Let $X_1$ and $X_2$ be a smooth proper varieties and $D_i \subset X_i$ be a smooth divisor for $i=1,2$ 
with an isomorphism $\phi \colon D_1 \rightarrow D_2$. 
Then there exists an SNC variety $X_0$ with a closed immersion $\iota_i \colon X_i \hookrightarrow X_0$ for $i=1,2$ in the Cartesian diagram 
\[
\xymatrix{
D_1 \ar[d]_{i_2 \circ \phi} \ar[r]^{i_1} & X_1 \ar[d]^{\iota_1} \\ 
X_2 \ar[r]^{\iota_2} & X_0
}, 
\] 
and we write $X_0=: X_1 \cup^{\psi} X_2$. 

Moreover, if $D_1$ is connected and $D_i \in |{-}K_{X_i}|$ for $i=1,2$, then we have $\omega_{X_0} \simeq \cO_{X_0}$. 
\end{prop}

\begin{proof}
See \cite[Proposition 2.1, Corollary 2.2]{Hashimoto:aa} and references therein. 
\end{proof}

\begin{defn}\label{defn:d-semistable}
Let $X$ be an SNC variety and $X= \bigcup_{i=1}^N X_i$ be the decomposition into its irreducible components. Let $D:= \Sing X = \bigcup_{i \neq j} (X_i \cap X_j)$ be the double locus and let  $I_{X_i}, I_D \subset \cO_X$ be the ideal sheaves of 
$X_i$ and $D$ on $X$. Let 
\[
\cO_D(X):= (\bigotimes_{i=1}^N I_{X_i} / I_{X_i} I_D)^{\ast} \in \Pic D
\]  
be the {\it infinitesimal normal bundle}  as in \cite[Definition 1.9]{MR707162}. 

We say that $X$ is {\it $d$-semistable} if $\cO_D(X) \simeq \cO_D$. 
If $X= X_1 \cup X_2$ for smooth varieties $X_1$ and $X_2$, then $\cO_D(X) \simeq \cN_{D/X_1} \otimes \cN_{D/X_2}$, where $\cN_{D/X_i}$ is the normal bundle of $D \subset X_i$ for $i=1,2$. 
\end{defn}

The following result on smoothings of an SNC variety is essential for the construction. 

\begin{thm}\label{thm:smoothingSNCCY}
Let $X$ be an $n$-dimensional proper SNC variety whose dualizing sheaf $\omega_X$ is trivial. 
Assume that $X$ is d-semistable.  
Then there exists a semistable smoothing $\phi \colon \cX \rightarrow \Delta^1$ of $X$ over a unit disk, that is, a proper surjective morphism such that $\cX$ is smooth, $X \simeq \phi^{-1}(0)$ and $\cX_t := \phi^{-1}(t)$ is smooth for $t \neq 0$.   
\end{thm}

\begin{proof}
This follows from \cite[Theorem 4.2]{MR1296351}, \cite[Corollary 7.4]{Chan:2019vv}. 
\end{proof}

\begin{rem}
For the construction of our examples, \cite[Theorem 4.2]{MR1296351} is enough although they assume that $H^{n-1}(X, \cO_X) =0$ and $H^{n-2}(X_i, \cO_{X_i}) =0$ for all irreducible component $X_i$ of $X$. 
Indeed, we can check that the SNC variety $X_0(m)$ as in Example \ref{eg:construction} satisfies the conditions since it is a union of two rational manifolds glued along a Calabi-Yau manifold. 

In \cite{Felten:2019ve} and \cite{Felten:2020vm}, generalizations of Theorem \ref{thm:smoothingSNCCY} are studied. 
\end{rem}

We shall use the following description of general rational elliptic surfaces as hypersurfaces in $\bP^2 \times \bP^1$.

\begin{prop}\label{prop:hypersurfaceP2xP1}
Let $S \subset \bP^2 \times \bP^1$ be a general hypersurface of bidegree $(3,1)$, that is a member of $|p_1^*\cO_{\bP^2}(3) \otimes p_2^*\cO_{\bP^1}(1)|$, where $p_1 \colon S \rightarrow \bP^2$ and $p_2 \colon S \rightarrow \bP^1$ are the projections. 
\begin{enumerate}
\item[(i)] $S$ is a rational elliptic surface with no $(-2)$-curve. Moreover, $p_1$ induces an anticanonical elliptic fibration 
and $p_2$ is the blow-up at 9 points which appear as intersection of two cubic curves.  
\item[(ii)] Let $C \subset S$ be an irreducible curve which is not a $(-1)$-curve. 
Then $C^2 \ge 0$ and $|C|$ is base point free. 
\end{enumerate}
\end{prop}

\begin{proof}
(i)
Note that $S$ is defined as 
$$(s F + t G =0) \subset \bP^2 \times \bP^1,$$
 where $[s:t] \in \bP^1$ is the coordinates 
and $F, G \in \bC[x_0, x_1, x_2]$ are general homogeneous polynomials of degree $3$. 
By this description, we see that $S$ is smooth and $p_2$ is the blow-up at  9 points $(F= G=0) \subset \bP^2$. 
We see that $- K_S = p_2^* \cO_{\bP^1}(1)$ and it induces an elliptic fibration, thus $S$ is a rational elliptic surface with a section. 

Since $-K_S$ is nef, an irreducible curve $C$ such that $C^2 <0$ is either a $(-1)$-curve or $(-2)$-curve. 
If $C \subset S$ is a $(-2)$-curve, then $C$ is $p_2$-vertical by $-K_S \cdot C =0$ and contained in a singular fiber. 
It is well-known that a generic pencil of cubic curves contains only irreducible curves (cf. \cite[Lemma 3.1]{MR2457523}), 
thus $S$ can not contain a $(-2)$-curve since a curve in a pencil can have singularities outside the base locus and all the members of $|{-}K_S|$ are irreducible.

\vspace{2mm}

\noindent(ii)
We see that $C^2 \ge 0$ since $K_S \cdot C \le 0$ and $C$ is neither a $(-1)$-curve nor a $(-2)$-curve. 
If $p_2(C)$ is a point, then $C$ is a fiber of $p_2$ since all fibers are irreducible and reduced. 
Thus $|C| = |{-}K_S|$ is a free linear system. 

If $p_2(C) = \bP^1$, then we see that $C^2 >0$ since $-K_S \cdot C >0$ and $C$ is not a $(-1)$-curve.  
We see that $$
h^0(S,\cO_S(C)) \ge \chi(S,\cO_S(C)) = \chi(S, \cO_S) + \frac{C(C-K_S)}{2} > 1+ \frac{C^2}{2} >1.   
$$
Thus $|C|$ has no fixed part and $C$ is nef and big. 
We also check that $|C|$ is free. Indeed, we have an exact sequence 
\[
H^0(S, \cO_S(C)) \rightarrow H^0(f, \cO_f(C)) \rightarrow H^1(S, \cO_S(C-f))
\] 
for a smooth fiber $f \in |{-}K_S|$, $H^1(S, \cO_S(C-f))= H^1(S, K_S+C) =0$ by the vanishing theorem and $f \cdot C \ge 2$ since $C$ is not a section. 
\end{proof}
	
We also need the following isomorphism of a rational elliptic surface $S$ induced by a quadratic transformation of $\bP^2$. 
(ii) and (iii) are technical, but essential in the construction of our Calabi-Yau manifolds. 	

\begin{prop}\label{prop:quadraticTrans}
Let $S \subset \bP^2 \times \bP^1$ be a general $(3,1)$-hypersurface. 
Let $p_1, \ldots, p_9 \in \bP^2$ be the points on which the birational morphism $\mu=p_1 \colon S \rightarrow \bP^2$ induced by the 1st projection 
is not an isomorphism. 
Let $E_i:= \mu^{-1}(p_i)$ for $i=1, \ldots, 9$ be $(-1)$-curves and $H:= \mu^*\cO_{\bP^2}(1)$. Then we have the following. 
\begin{enumerate}
\item[(i)] For $1\le i<j<k \le 9$, there exist a hypersurface $S_{ijk} \subset \bP^2 \times \bP^1$ and an isomorphism $\phi_{ijk} \colon S \rightarrow S_{ijk}$ over $\bP^1$
such that $$\phi_{ijk}^* (H_{ijk}) = 2H - E_i - E_j - E_k, $$ where $H_{ijk}$ is the pull-back of $\cO_{\bP^2}(1)$ to $S_{ijk}$ by the 1st projection. 
\item[(ii)] For a positive integer $m$, there exist a hypersurface $S_m \subset \bP^2 \times \bP^1$ and an isomorphism $\phi_m \colon S \rightarrow S_m$ over $\bP^1$ such that 
\begin{equation}\label{eq:phi-mpullback}
\phi_m^* H_{S_m} = (27m^2+1)H - (9m^2-3m)F_1 - 9m^2 F_2 - (9m^2+3m) F_3, 
\end{equation}
where $H_{S_m}$ is the pull-back of $\cO_{\bP^2}(1)$ to $S_m$ and $F_i := E_{3i-2} + E_{3i-1} + E_{3i}$ for $i=1,2,3$.  
\item[(iii)] In (ii), the divisor $L_m:= H + \phi_m^* H_{S_m} + mK_S$ is ample and free on $S$. 
\end{enumerate}
\end{prop}

\begin{rem}
In (iii), the essential part used later is that the divisor is effective and can be written as a sum of smooth curves.  
However, it is nice to show freeness since then the linear system contains a smooth irreducible member. 
By this property, we can determine the 2nd Betti number of our Calabi-Yau manifolds obtained as smoothings. 
\end{rem}

\begin{proof}
\noindent(i) We see that $p_1, \ldots, p_9 \in \bP^2$ is in ``Cremona general position'' in the sense of \cite[pp.1178]{MR2457523}, 
thus can consider the quadratic transformation $q_{ijk}\colon \bP^2 \dashrightarrow \bP^2$ at $p_i, p_j, p_k$. 
This $q_{ijk}$ induces an isomorphism $\phi_{ijk} \colon S \rightarrow S_{ijk}$ onto a hypersurface  $S_{ijk} \subset \bP^2 \times \bP^1$. 
By the construction, we see that $\phi_{ijk}$ is induced by the birational transformation $q_{ijk} \times \id \colon \bP^2 \times \bP^1 \dashrightarrow \bP^2 \times \bP^1$, thus $\phi_{ijk}$ is an isomorphism over $\bP^1$. 

Let $E'_i, E'_j, E'_k \in \Pic S_{ijk}$ be the images of $H-E_j-E_k, H- E_i - E_k, H- E_i - E_j \in \Pic S$ by $\phi_{ijk}$. 
For $l \neq i,j,k$, let $E'_l:= \phi_{ijk}(E_l) \subset S_{ijk}$. 
We identify $\Pic S \simeq \bZ^{10}$ and $\Pic S_{ijk} \simeq \bZ^{10}$ by the basis $(H, E_1, \ldots, E_9)$ and $(H_{ijk}, E'_1, \ldots, E'_9)$. Here $\bZ^{10}$ is the lattice with a bilinear form $$(a, b_1, \ldots, b_9) \cdot (a', b'_1, \ldots, b'_9):= aa' - \sum_{l=1}^9 b_l b'_l. $$ 
Let $h, e_i, e_j, e_k \in \bZ^{10}$ be the images of $H, E_i, E_j, E_k$ via the identification $\Pic S \rightarrow \bZ^{10}$, that is, 
\[
h:= (1,0, \ldots ,0), \\
e_1:= (0,1, 0, \ldots, 0), \cdots, e_9:= (0,0,\ldots, 0,1) \in \bZ^{10}.
\] 
Then we see that $\phi_{ijk}^* \colon \Pic S_{ijk} \rightarrow \Pic S$ is the reflection orthogonal to $H- E_i - E_j - E_k$, that is, it induces $\phi_{ijk}^* \colon \bZ^{10} \rightarrow \bZ^{10}$ determined by
$$\phi_{ijk}^*(x) = x+ (x \cdot \alpha_{ijk}) \alpha_{ijk}, $$ 
where $\alpha_{ijk}:=h-e_i - e_j - e_k \in \bZ^{10}$.  
Hence we obtain the required equality by substituting $x=h$. 

\vspace{2mm}

\noindent(ii) 
Let $\psi_S:= \phi_{789} \circ \phi_{456} \circ \phi_{123} \colon S \rightarrow S'$ be a composite of three isomorphisms as in (i), that is, 
$\phi_{123}$ is the quadratic transformation at $\{ p_1, p_2, p_3 \}$, and $\phi_{456}$ and $\phi_{789}$ are quadratic transformation at 
the images of $\{ p_4,p_5,p_6\}$ and $\{p_7, p_8, p_9 \}$ respectively. 
Then let $$\psi_1:=\psi_{S'} \circ \psi_S \colon S \rightarrow S'', $$ where $\psi_{S'}$ is the same operation on $S'$ and let $S_1:= S''$ (and $S_0:=S$). 
We can perform this operation for any positive integer $i$ and construct an isomorphism 
$\psi_i \colon S_{i-1} \rightarrow S_i$ as a composite of six quadratic transformations. 
Now let $\phi_m:= \psi_m \circ \cdots \circ \psi_1 \colon S \rightarrow S_m$. 

On each surface $T=S_i$, we have an isomorphism 
$\Pic T \rightarrow \bZ^{10}$ determined by the basis $(H_{T}, E_{T,1}, \ldots,  E_{T, 9})$, where $H_T$ is the pull-back of $\cO_{\bP^2}(1)$ and $E_{T,i}$'s are the exceptional divisors. 
Note that the labelling for the exceptional divisors are determined as (i). 
We also use the same symbol for isomorphisms of the Picard group and $\bZ^{10}$, that is, 
for $\psi_1 \colon S \rightarrow S'$, we write $\psi_1^* \colon \Pic S' \rightarrow \Pic S$ and $\psi_1^* \colon \bZ^{10} \rightarrow \bZ^{10}$, for example. 
We are reduced to show the following to obtain the equality (\ref{eq:phi-mpullback}). 

\begin{claim} Let $h, e_1, \ldots, e_9 \in \bZ^{10}$ be the elements corresponding to 
$(H, E_1, \ldots, E_9)$ or $(H_{S_m}, E_{S_m, 1}, \ldots, E_{S_m, 9})$ as before. Let $f_i := e_{3i-2} + e_{3i-1} + e_{3i}$ for $i=1,2,3$. Then we have 
\begin{equation}\label{eq:pullbackclaim}
\phi_m^*(h) = (27m^2+1)h - (9m^2-3m)f_1 - 9m^2 f_2 - (9m^2+3m) f_3. 
\end{equation}
\end{claim}

\begin{proof}[Proof of Claim]
We check the required equality by induction on $m$. 
Recall that $\phi_{ijk}^* \colon \bZ^{10} \rightarrow \bZ^{10}$ is the reflection for $\alpha_{ijk} = h-e_i -e_j -e_k \in \bZ^{10}$. 
Then we compute 
\begin{multline*}
\psi_S^*(h) = \phi_{789}^*(\phi_{456}^*(\phi_{123}^*(h)) ) =\phi_{789}^*(\phi_{456}^*(2h-f_1) ) \\
 =\phi_{789}^*(4h-f_1-2f_2) =  8h-f_1-2f_2 -4f_3. 
\end{multline*}
 Then we compute similarly
\begin{equation*}
\psi_1^*(h) = \phi_{789}^*(\phi_{456}^*(\phi_{123}^*(8h-f_1 -2f_2 -4f_3)) ) = 28h-6f_1-9f_2 -12f_3, 
\end{equation*}
thus obtain the equality for $\phi_1 = \psi_1$. 
Suppose that we have the equality (\ref{eq:pullbackclaim}) for $\phi_m$. 
By a similar computation, we obtain 
\begin{multline*}
\phi_{m+1}^*(h) = \psi_{m+1}^* \phi_m^*(h) \\ 
= \psi_{m+1}^*((27m^2+1)h - (9m^2-3m)f_1 - 9m^2 f_2 - (9m^2+3m) f_3) \\
= (27(m+1)^2+1)h - (9(m+1)^2-3(m+1)) f_1 - 9(m+1)^2 f_2 - (9(m+1)^2 + 3(m+1)) f_3. 
\end{multline*}
Thus we obtain the claim by induction. 
\end{proof}
This finishes the proof of (ii). 

\vspace{2mm}

\noindent(iii) 
We have $L_m \cdot (-K_S) = 6 >0$ and 
\begin{multline*}
L_m^2 = (H + \phi_m^*H_{S_m})^2 + 2m (K_S \cdot (H+ \phi_m^*H_{S_m}))\\ 
 = (2+2(27m^2+1)) -12m
= 54m^2 -12m +4 >0. 
\end{multline*}
By these and the Riemann-Roch formula, 
we see that $L_m$ is effective. 

Since $S$ is general, it is enough to show $L_m \cdot C >0$ for all $(-1)$-curve $C$.  
We write $C= \alpha H - \sum_{i=1}^9 \beta_i E_i$ for some integers $\alpha, \beta_1, \ldots, \beta_9$ 
such that $\alpha^2 - \sum_{i=1}^9 \beta_i^2 =-1$ and $3\alpha - \sum_{i=1}^9 \beta_i =1$. Let $\gamma_i:= \beta_{3i-2}+\beta_{3i-1} + \beta_{3i}$ 
for $i=1,2,3$. 
We may assume that $\alpha \le m$ since, if $\alpha >m$, then we have 
 $$L_m \cdot C \ge (H+mK_S)\cdot C = \alpha-m >0. $$ 
 We may also assume that $\beta_i \ge 0$ for all $i$ since we have 
 \[
 L_m \cdot E_i = \phi_m^*(H_{S_m}) \cdot E_i + mK_S \cdot E_i \ge 
 (9m^2 -3m) -m >0. 
 \]
Then we compute 
\begin{multline*}
L_m \cdot C = \alpha(27m^2+2) -(9m^2-3m)\gamma_1 - 9m^2 \gamma_2 - (9m^2+3m)\gamma_3 -m \\ 
= \alpha(27m^2+2) -(9m^2-3m)(3\alpha-1) - 3m(\gamma_2 + 2 \gamma_3) -m \\
 =(9m+2) \alpha +(9m^2-3m) -3m(\gamma_2 + 2 \gamma_3)-m  \\ 
 \ge (9m+2) \alpha +(9m^2-3m) -3m(2(3\alpha-1)) -m \\
 = -9m\alpha + 2 \alpha +9m^2 +2m = 9m(m - \alpha)+ 2m + 2 \alpha \ge 2m + 2 \alpha >0,  
\end{multline*}
where we used $\gamma_2 + 2\gamma_3 \le 2(3\alpha-1)$ for the 1st inequality and $m \ge \alpha$ for the 2nd inequality.  
Thus we see that $L_m$ is ample by the Nakai--Moishezon criterion. 
Since $L_m \cdot (-K_S) = 6$, we see that $L_m|_{F}$ is free for a general smooth element $F \in |{-}K_S|$. 
By this and the exact sequence $$0 \rightarrow \cO_S(K_S +L_m) \rightarrow \cO_S(L_m) \rightarrow \cO_F(L_m) \rightarrow 0,$$ 
we check that $|L_m|$ is free as in the proof of Proposition \ref{prop:hypersurfaceP2xP1}(ii). 
\end{proof}

We have the following description of Calabi-Yau manifolds of ``Schoen type'' arising from general rational elliptic surfaces (\cite{MR923487}). 
The author learned (i) in the following proposition in \cite{MOF2020} when $n=2$ (cf. \cite[Section 2]{MR1672045}).  

\begin{prop}\label{prop:SchoenCY3}
Let $S \subset \bP^2 \times \bP^1$ be a general $(3,1)$-hypersurface and $T \subset \bP^1 \times \bP^n$ a 
general $(1,n+1)$-hypersurface for some $n \ge 2$. 
Let $X_1:= \bP^2 \times T$ and $X_2:= S \times \bP^n$ be divisors in $\bP^2 \times \bP^1 \times \bP^n$ 
and $X_{12}:= X_1 \cap X_2$.  Let $p_S \colon X_{12} \rightarrow S$ and $p_T \colon X_{12} \rightarrow T$ be 
the surjective morphisms induced by the projections.  
\begin{enumerate}
\item[(i)] $X_{12}$ is a Calabi-Yau $(n+1)$-fold and there is a natural isomorphism $$\varphi \colon S \times_{\bP^1} T \rightarrow X_{12}, $$ where the fiber product is defined by the projections 
$\phi_S \colon S \rightarrow \bP^1$ and $\phi_T \colon T \rightarrow \bP^1$.  
\item[(ii)](cf. \cite[Corollary 3.2]{MR1240604}) We have  $$\Pic X_{12} \simeq (\Pic S \oplus \Pic T)/ \bZ(-K_S, K_T). $$
Thus $\Pic X_{12} \simeq \bZ^{19}$ when $ n=2$ and $\Pic X_{12} \simeq \bZ^{11}$ when $n \ge 3$. 
\end{enumerate}
\end{prop}
	
\begin{proof}
\noindent(i) This follows from properties of fiber products. 

\vspace{2mm}

\noindent(ii) By the same argument as the 1st paragraph of the proof of \cite[Proposition 1.1]{MR1093334}, we see that the naturally induced homomorphism 
$$p_S^* \oplus p_T^* \colon \Pic S \oplus \Pic T \rightarrow \Pic X_{12}$$ is surjective. 
By the same argument as the proof of \cite[Proposition 3.1]{MR1240604}, we can write $(L_1, L_2) \in \Ker (p_S^* \oplus p_T^*)$  as 
$$(L_1, L_2) = (A_1, A_2) + m(-K_S, K_T)$$ for some numerically trivial 
$A_1 \in \Pic S$, $A_2 \in \Pic T$ and $m \in \bZ$. Since $H^1(S, \cO_S) = H^1(T, \cO_T)=0$, we see that 
$A_1$ and $A_2$ are linearly trivial. Thus we see that $\Ker (p_S^* \oplus p_T^*) = \bZ(K_S, -K_T)$ and obtain the required isomorphism. 

Since the projection morphism $T \rightarrow \bP^n$ is the blow-up along the intersection $(F=G=0) \subset \bP^n$ of general divisors $(F=0), (G=0)$ of degree $n+1$, we see that $\Pic T \simeq \bZ^{10}$ if $n=2$ and $\Pic T \simeq \bZ^2$ if $n \ge 3$. 
Thus we obtain the latter statement. 
\end{proof}	

%\begin{rem}
%The $X_{12}$ is also studied in \cite{Suzuki:2021tv}. He also shows that $X_{12}$ is simply connected. 
%\end{rem}
%
%	
%\begin{rem}
%We can also consider the case $n=1$. In this case, $T \subset \bP^1 \times \bP^1$ is also $\bP^1$ and $X_{12}$ is a double cover of $S$. We compute that $\Pic X_{12} \simeq \bZ^{10}$.  
%\end{rem}	
	
\section{Construction of examples}	

We first explain the construction of examples $X(m)$ by smoothing SNC varieties $X_{0}(m)$ in the following.

%\begin{eg}
%We shall construct an SNC Calabi-Yau variety whose infinitesimal normal bundle is ``arbitrarily positive'' as follows.  
%
%Let $S \subset \bP^2 \times \bP^1$ be a general hypersurface as in Proposition \ref{prop:hypersurfaceP2xP1}. 
%Let $X:= S \times \bP^2 \subset \bP^2 \times \bP^1 \times \bP^2$ and $D \in |-K_X|$ be a general smooth member. 
%Note that $D = X \cap \bP^2 \times T$ for a general hypersurface $T \subset \bP^1 \times \bP^2$ of bidegree $(1,3)$ 
%since the restriction homomorphism $H^0(\cO_{\bP^2 \times \bP^1 \times \bP^2}(0,1,3)) \rightarrow H^0(X, \cO_X(-K_X))$ 
%is surjective. Here $\cO_{\bP^2 \times \bP^1 \times \bP^2}(a_1,a_2,a_3):= \otimes_{i=1}^3 \pr_i^*\cO(a_i)$, where $\pr_i$ is 
%the $i$-th projection. 	
%
%Thus $D \simeq S \times_{\bP^1} T$ for the anti-canonical elliptic fibrations $\phi_S \colon S \rightarrow \bP^1$ and 
%$\phi_T \colon T \rightarrow \bP^1$ and $D$ is a Calabi-Yau 3-fold. 
%\end{eg}

\begin{eg}\label{eg:construction}
Let $m, n$ be positive integers. 
Let $S \subset \bP^2 \times \bP^1$ be a general $(3,1)$-hypersurface 
and $S_m \subset \bP^2 \times \bP^1$ be the hypersurface in Proposition \ref{prop:quadraticTrans}(ii) 
with the isomorphism $\phi_m \colon S \rightarrow S_m$ over $\bP^1$. 

Let $T \subset \bP^1 \times \bP^n$ be a general $(1,n+1)$-hypersurface and 
 $$Y_1:= \bP^2 \times T \subset \bP^2 \times \bP^1 \times \bP^n.$$ 
 Let $D_1:= Y_1 \cap (S \times \bP^n) \subset Y_1$. 
 Then there is a natural isomorphism $\varphi_1 \colon D_1 \rightarrow S \times_{\bP^1}  T$ as in Proposition \ref{prop:SchoenCY3}(i). 
 Note that $D_1 \in |{-}K_{Y_1}|$ and the normal bundle $\cN_{D_1/Y_1}   \simeq p_S^* \cO_S(3h+f)$, 
where $p_S \colon D_1 \rightarrow S$ is the projection, $f \in |{-}K_S|$ and $h:= \mu_S^* \cO_{\bP^2}(1)$ for the birational morphism $\mu_S \colon S \rightarrow \bP^2$ 
induced by the projection.

 Let $Y_2:= \bP^2 \times T$ and $D_2 := Y_2 \cap (S_m \times \bP^n) \subset Y_2$. 
 Then there is a natural isomorphism $\varphi_2 \colon D_2 \rightarrow S_m \times_{\bP^1}  T$ as above. 
Let $\Phi_m \colon D_1 \rightarrow D_2$ be the isomorphism which fits in the following diagram: 
\[
\xymatrix{
D_1 \ar[d]^{\varphi_1} \ar[r]^{\Phi_m} & D_2 \ar[d]^{\varphi_2}  \\
S \times_{\bP^1} T \ar[r]^{\phi_m \times \id} & S_m \times_{\bP^1} T. 
}
\]
For $i=1, \ldots, m$, let $F_i:= p_S^{-1}(f_i) \subset D_1$ for  a smooth general member $f_i \in |{-}K_S|$. 
Let $\Gamma_m:= p_S^{-1}(C_m)$ for a smooth member 
$$C_m \in |3H + \phi_m^*(3H_{S_m}) + (m-2)K_S| = |L_m +2(H+\phi_m^*H_{S_m}) + (-2K_S)|$$ 
of an ample and free linear system guaranteed by Proposition \ref{prop:quadraticTrans}(iii).
Now let $\nu_1 \colon \tilde{Y}_1 \rightarrow \bP^2 \times T =Y_1$ be the blow-up along $F_1, \ldots, F_m$, and $\nu_2 \colon X_1 \rightarrow \tilde{Y}_2$ be the blow-up along the strict transform $\tilde{\Gamma}_m \subset \tilde{Y}_1$ of $\Gamma_m \subset D_1$. 
Thus we have a composition $\mu := \nu_1 \circ \nu_2 \colon X_1 \rightarrow Y_1$. 
Let $X_2:= Y_2 = \bP^2 \times T$. 
Let $\tilde{D}_1 \subset X_1$ be the strict transform of $D_1$ and $\mu_{D_1} \colon \tilde{D}_1 \rightarrow D_1$ be the induced isomorphism. 
Then we can glue $X_1$ and $X_2$ along the composition isomorphism 
$$\Psi_m \colon \tilde{D}_1 \xrightarrow{\mu_{D_1}} D_1 \xrightarrow{\Phi_m} D_2$$ 
and construct an SNC variety $$X_0(m):= X_0:= X_1 \cup^{\Psi_m} X_2$$ by Proposition \ref{prop:gluingSNC}. 
We check that $X_0$ is d-semistable since we have 
\[
\cN_{D_1/Y_1} \otimes \Phi_m^*\cN_{D_2/Y_2} \simeq \cO_{D_1}(p_S^*(3(H+\phi_m^*H_{S_m}) +2f)) \simeq \cO_{D_1}(F_1+ \cdots + F_m + \Gamma_m)
\] 
and the blow-up centers of $\mu$ are chosen so that this becomes trivial. 
We also see that $\omega_{X_0} \simeq \cO_{X_0}$ since $\tilde{D_1} \in |{-}K_{X_1}|$ and $D_2 \in |{-}K_{X_2}|$. 
Thus we can apply Theorem \ref{thm:smoothingSNCCY} and construct a semistable smoothing $\cX(m) \rightarrow \Delta^1$. 
Let $X(m)$ be its general smooth fiber. Thus we obtain a compact complex manifold $X(m)$. We also write $X:= X(m)$ for short in the following. 
\end{eg}

\subsection{Properties of the smoothings}

We have the following basic properties of $X(m)$ in Example \ref{eg:construction}.

 \begin{prop}\label{prop:X_mproperty} The above $X = X(m)$ satisfies the following. 
 \begin{enumerate}
 \item[(i)] The Hodge to de Rham spectral sequence $H^q(X, \Omega_X^p) \Rightarrow H^{p+q}(X, \bC)$ degenerates at $E_1$. 
\item[(ii)] We have $H^i(X, \cO_X)=0$ and $H^0(X, \Omega_X^i) =0$ for $0 < i < \dim X$. 
We also have $\omega_X \simeq \cO_X$, thus $X$ is a Calabi-Yau manifold. 
\item[(iii)] The 2nd betti number is $b_2(X) = m+\rho_T$, where $\rho_T:= \rk \Pic T$. 
\end{enumerate}
 \end{prop}
 
 \begin{proof}
 \noindent(i) This is \cite[Corollary 11.24]{MR2393625}. 

\vspace{2mm} 
 
\noindent(ii) Let $X_{12}:= X_1 \cap X_2$. We compute $H^i(X_0, \cO_{X_0}) =0$ for $0<i < \dim X$ by the exact sequence 
\[
\cdots \rightarrow H^{i-1}(X_{12}, \cO_{X_{12}}) \rightarrow H^i(X_0, \cO_{X_0}) \rightarrow \bigoplus_{j=1}^2 H^i(X_j, \cO_{X_j})  \rightarrow \cdots.    
\]
By this and the upper semi-continuity theorem, we obtain $H^i(X, \cO_X) =0$ for $0< i < \dim X$. 
Since we have an exact sequence 
$$
H^0(X, \cO_X) \rightarrow H^0(X, \cO_X^*) \rightarrow H^1(X, \bZ) \rightarrow H^1(X, \cO_X) 
$$
from the exponential exact sequence, we see that $H^1(X, \bZ) =0$. By this and (i), we obtain $H^0(X, \Omega_X^1) =0$. 
%We also see that $H^0(X, \Omega_X^2)=0$ as \cite[Remark 3.8]{Hashimoto:aa} by $H^1(X, \cO_X^*) \simeq H^2(X, \bZ)$. 

\begin{claim}
We have $H^0(X, \Omega^i_X) =0$ for $2 \le i \le \dim X-1$. 
\end{claim}

\begin{proof}[Proof of Claim]
For the semistable smoothing $\cX \rightarrow \Delta^1$ and $i \ge 0$, we have the locally free sheaf
\[
\Lambda_{X_0}^i := \Omega^i_{\cX/ \Delta^1}(\log X_0)|_{X_0}   
\]
which is defined in \cite[pp.92]{MR707162}. 
%($\Lambda^1_{X_0}$ can be regarded as the differential sheaf for the ``log smooth morphism'' $X_0^{\dag} \rightarrow \bC^{\dag}$ over the standard log point $\bC^{\dag}$). 
  It is enough to show $H^0(\Lambda^i_{X_0}) =0$ for $2 \le i \le \dim X-1$ since the rank $h^0(\Omega^i_{\cX/ \Delta^1}(\log X_0)|_{X_t})$ is upper-semicontinuous for the fibers $X_t$ over $t \in \Delta^1$. 
By applying \cite[Proposition 3.5]{MR707162} to $\Lambda_{X_0}^i$ for $X_0 = X_1 \cup X_2$, we have an exact sequence 
\[
0 \rightarrow V_0 \rightarrow \Lambda_{X_0}^i \rightarrow V_1/V_0 \rightarrow 0, 
\]
where $V_0$ and $V_1/V_0$ are described as 
\[
V_0 \simeq \Ker (\Omega_{X_1}^i \oplus \Omega_{X_2}^i \xrightarrow{(\iota_1^*, -\iota_2^*)} \Omega_{X_{12}}^i), \ \ V_1/V_0 \simeq \Omega_{X_{12}}^{i-1}. 
\]
By this and $H^0(X_i, \Omega_{X_j}^i)=0=H^0(X_{12}, \Omega_{X_{12}}^{i-1})$ for $i=2, \ldots, \dim X -1$ and $j=1,2$, we obtain $H^0(V_0) =0$ and $H^0(V_1/V_0) =0$. 
Thus we obtain $H^0(\Lambda_{X_0}^i)=0$ for $2 \le i \le \dim X -1$ by the above exact sequence. 
\end{proof}

We also see that $\omega_X \simeq \cO_X$ as in the proof of \cite[Theorem 3.4]{Hashimoto:aa}. Hence we see that $X$ is a Calabi-Yau manifold in our sense. 

\vspace{2mm}

\noindent(iii) Note that $\Pic X_0 \simeq H^2(X_0, \bZ)$ and $\Pic X \simeq H^2 (X, \bZ)$ by $H^i(X_0, \cO_{X_0}) =0$ and $H^i(X, \cO_X) =0$ for $i=1,2$. 
We compute $b_2(X_0) = m+\rho_T +1$ as follows. 
We have an exact sequence 
\[
0 \rightarrow \Pic X_0 \rightarrow \Pic X_1 \oplus \Pic X_2 \xrightarrow{R} \Pic X_{12}, 
\]
where $R= (\iota_1^*, -\iota_2^*)$ for the closed immersion $\iota_i \colon X_{12} \hookrightarrow X_i$ for $i=1,2$. 
By using the isomorphism
\[
\Pic X_{12} \simeq (\Pic S \oplus \Pic T)/ \bZ(K_S, -K_T) 
\] as in Proposition \ref{prop:SchoenCY3}, we see that the image $\Image R \subset \Pic X_{12}$ is generated by the image of $\Pic T$ and $p_S^*(H), p_S^*(\phi_m^*(H_{S_m}))$. 
Thus we see that $$\Image R \simeq \bZ^{\rho_T+2}.$$   
Since $\Pic X_1 \simeq \bZ^{1+ \rho_T+m+1}= \bZ^{m+\rho_T+2}$ and $\Pic X_2 \simeq \bZ^{\rho_T+1}$, 
we see that $$\Pic X_0 \simeq \bZ^{m+\rho_T +1}$$ by the above exact sequence, thus obtain $b_2(X_0) = m+ \rho_T +1$.

Then, by the Clemens map $\gamma \colon X \rightarrow X_0$ (cf. \cite{MR444662}, \cite[Theorem 2.9]{Hashimoto:aa}), we have the exact sequence 
\[
\bZ \simeq H^0(X_0, R^1 \gamma_* \bZ) \rightarrow H^2(X_0, \bZ) \rightarrow H^2(X, \bZ) \rightarrow H^1(R^1 \gamma_* \bZ) =0.  
\]
Thus we see that $H^2(X, \bZ) \simeq \bZ^{m+\rho_T}$ as in \cite[Claim 3.6(ii)]{Hashimoto:aa}. 
 \end{proof}

The following lemma is useful to see the non-projectivity of $X$ and compute the algebraic dimension of $X$ 

\begin{lem}\label{lem:X0linebundle}
Let $X_0=X_0(m)$ be the SNC Calabi-Yau variety as in Example \ref{eg:construction}. Let $N:= \dim X_0$.  
Let $\cL_0 \in \Pic X_0$ be a line bundle such that $\cL_i:= \cL_0|_{X_i}$ is effective for $i=1,2$. 
We may write 
$$
\cL_1= \mu^* (\cO_{\bP^2}(a) \boxtimes \cO_T(H_1) ) \otimes \cO_{X_1}(\sum_{j=1}^{m} b_j E_j + c F),$$
$$
 \cL_2= \cO_{\bP^2}(a') \boxtimes \cO_{T}(H_2), 
$$
 where $E_j:= \mu^{-1}(F_j) \subset X_1$ for $j=1, \ldots , m$ and $F:= \mu^{-1}(\Gamma_m)$ are the exceptional divisors of $\mu \colon X_1 \rightarrow \bP^2 \times T$ and 
 $H_1, H_2 \in \Pic T$. 

Then we have $a=a'=0$ and $\kappa(\cL_0) \le N -2$. 
\end{lem}

\begin{proof}
Note that $\cL_1|_{\tilde{D}_1} \simeq \Psi_{m}^* \cL_2|_{D_2}$, where $\Psi_m \colon \tilde{D}_1 \rightarrow D_2$ is the isomorphism used to construct $X_0$. 
We have a natural surjection 
$$
\pi_S \colon \Pic \tilde{D}_1 \xrightarrow{\simeq} (\Pic S \oplus \Pic T)/ \bZ(K_S, -K_T) \xrightarrow{\pr} \Pic S/\bZ(K_S)
$$ 
by Proposition \ref{prop:SchoenCY3}(ii), where $\pr$ is the projection.   
We see that 
\[
\pi_S(\cL_1|_{\tilde{D}_1}) = [a H + c(3(H+ H'))] = (a+3c)[H] + 3c[H'] \in \Pic S/\bZ (K_S),   
\]
where $[\cdot]$ means the image of $\Pic S \rightarrow \Pic S /\bZ K_S$ and $H':= \phi_m^*(H_{S_m})$. 
We also see that 
\[
\pi_S(\Psi_m^* \cL_2|_{D_2}) = a' [H']. 
\]
By comparing the above two terms, we see that $a' = 3c$ and $a+3c =0$ since $[H]$ and $[H']$ are linearly independent. 
Since $a, a' \ge 0$, we obtain $a=a'=0$. This implies that $\kappa(\cL_i) \le \dim T$ for $i=1,2$, thus $\kappa(\cL_0) \le \dim T=N-2$.  
\end{proof} 
 
\begin{prop}
For $m>0$, $X=X(m)$ is not projective. 
Moreover, we have $$a(X) =n = \dim X -2$$ for a very general $t \in \Delta^1$, where $a(X)$ is the algebraic dimension. 
\end{prop} 
	
\begin{proof}
Let $\pr_T \colon \bP^2 \times T \rightarrow T$ be the projection. 
We see that, for a very ample $H_T \in \Pic T$, 
the line bundles $H_1:= \mu^*(\pr_T^*\cO_T(H_T)) \in \Pic X_1$ and $H_2 := \pr_T^* \cO_T(H_T)$ induce a line bundle $H_0 \in \Pic X_0$ 
such that $H_0|_{X_i} \simeq H_i$. Since we have an isomorphism $\Pic X_0 \simeq \Pic \cX$ as in \cite{Hashimoto:aa}, $H_0$ induces a line bundle $H_t$ and this induces 
a fiber space $X \rightarrow T$. Its general fiber is a K3 surface since the general fiber of $X_0 \rightarrow T$ is an SNC surface which is a union of $\bP^2$ and 
its blow-up at $18$ points. Thus the fiber space $X \rightarrow T$ is a K3 fibration. 

We see that there is no line bundle $L_t$ on $X$ such that $\kappa(L_t) \ge \dim T +1$ by the same argument as \cite[Proposition 3.19(iii)]{Hashimoto:aa}. 
Indeed, if such a line bundle exists, then there exists $L_0 \in \Pic X_0$ such that $L_0|_{X_i}$ is effective for $i=1,2$ and $\kappa(L_0) \ge \dim T +1$. 
This contradicts Lemma \ref{lem:X0linebundle}.  
\end{proof}	

We can also compute the following topological invariants of $X$.

\begin{prop}\label{prop:pi1Euler}
\begin{enumerate}
\item[(i)] $X$ is simply connected. 
\item[(ii)] The topological Euler number of $X$ is 
\begin{equation}\label{eq:Euler}
e(X)= (\gamma_m-12) \frac{(-n)^{n+1}+n^2 +2n}{n+1} +24(n+1)(-n)^n, 
\end{equation}
where $\gamma_m:= -18(27m^2-2m+5)$.  
\end{enumerate}
\end{prop}

\begin{proof}
\noindent(i) We show this by following the proof of \cite[Proposition 3.10]{Hashimoto:aa}. 

%We first see that $\pi_1(X_0) = \{1 \}$ by van Kampen's theorem. 

%Note that it is enough to show that $\pi_1(X)$ is abelian since $H^1(X, \bZ) =0$ by Proposition \ref{prop:X_mproperty}(i). 
As in \cite[Proposition 3.10]{Hashimoto:aa}, we see that 
\begin{equation}\label{eq:pi1Xformula}
\pi_1(X) \simeq \pi_1(X_1') \ast_{\pi_1(\tilde{X}_{12})} \pi_1(X_2'), 
\end{equation}
where $X_i':= X_i \setminus X_{12}$ for $i=1,2$ and  $\tilde{X}_{12}:= \gamma^{-1}(X_{12})$ for the Clemens map $\gamma \colon X \rightarrow X_0$. 
We see that $\pi_1(X_2') = \{1 \}$ by the Gysin exact sequence 
\[
H_2(X_2, \bZ) \rightarrow H_0(X_{12}, \bZ) \rightarrow H_1(X_2', \bZ) \rightarrow H_1(X_2, \bZ). 
\]
Indeed, for a section $C \subset T$ of the fiber space $T \rightarrow \bP^1$, the class $[\{p\} \times C] \in H_2(X_2, \bZ)$ is sent to a generator of $H_0(X_{12})$. For example, a fiber of $T \rightarrow \bP^n$ over the exceptional locus can be taken as $C$. 
We also see that $\pi_1(X_1') = \{ 1\}$ as in \cite{Hashimoto:aa} (In fact, the argument is easier since $\pi_1(X_2')= \{1 \}$). 
By these and (\ref{eq:pi1Xformula}), we see that $\pi_1(X) = \{ 1\}$. 

\vspace{2mm}

\noindent(ii) As in the proof of  \cite[Claim 3.7]{Hashimoto:aa}, 
we see that $$e(X) = e(X_1) + e(X_2) - 2e(X_{12})$$ by the Mayer-Vietoris exact sequence and the Clemens map. 
Note that $T \rightarrow \bP^1$ has singular fibers with only one nodes over points $q_1, \ldots , q_{d_n} \in \bP^1$, 
where 
$$d_n:=  (n+1)n^n$$ is the degree of the discriminant hypersurface in $|\cO_{\bP^n}(n+1)|$ (cf. \cite[Lemma 2.1]{MR2569617}). 
Note also that a smooth hypersurface $H_{n+1} \subset \bP^n$ has the Euler number 
\[
e(H_{n+1}) = (n+1) + \frac{(-n)^{n+1} -1}{n+1}= \frac{(-n)^{n+1}+n^2 + 2n}{n+1} =: \sigma_n. 
\]
Thus we compute 
\[
e(T)= e(\bP^1) e(H_{n+1}) + d_n(-1)^n = 2 \sigma_n + \delta_{n}, 
\] 
where we put $\delta_{n}:= (-1)^{n} d_{n}$. 
We compute that $$e(X_2) = e(\bP^2 \times T) = e(\bP^2) e(T)=3 (2\sigma_n+\delta_n) = 6\sigma_n + 3\delta_n. $$
To compute $e(X_1)$, note that $X_1 \rightarrow \bP^2 \times T$ is the blow-up along $F_1, \ldots, F_m$ and (the strict transform of) $\Gamma_m$. 
Note that $e(F_i) =0$ since $F_i$ is a product of an elliptic curve and a Calabi-Yau hypersurface. 
Thus we see that 
$$
e(X_1) = e(\bP^2 \times T) + e(\Gamma_m) = 3 (2\sigma_n + \delta_n) + e(\Gamma_m). 
$$
Note also that $\Gamma_m \rightarrow C_m$ is a Calabi-Yau fibration and the discriminant locus consists of $18 d_n$ points since 
$$
C_m \cdot (-d_n K_S) = d_n(-K_S \cdot (3H+3\phi^*H)) =18 d_n.  
$$
Indeed, $\Gamma_m \rightarrow C_m$ is singular at the intersection of $C_m$ and the discriminant locus $D_S \subset S$ of $X_{12} \rightarrow S$, 
and $D_S$ consists of smooth fibers over the $d_n$ points which is the discriminant locus of $T \rightarrow \bP^1$. We also check 
\[
e(C_m)=-(K_S+C_m)\cdot C_m = -18(27m^2 -2m +5)=: \gamma_m. 
\]
Thus we compute 
$$e(\Gamma_m)  = e(C_m) \cdot e(H_{n+1}) + 18d_n (-1)^n= \gamma_m \sigma_n +18\delta_n$$ and 
\[
e(X_1) = (6 \sigma_n + 3 \delta_n) + (\gamma_m \sigma_n +18 \delta_n) = (\gamma_m +6) \sigma_n + 21 \delta_n. 
\]
To compute $e(X_{12})$, note that $X_{12} \rightarrow \bP^1$ has a fiber with non-zero Euler number only at the discriminant locus 
$p_1, \ldots, p_{12}$ of $S \rightarrow \bP^1$. 
Then we compute 
\[
e(X_{12})= 12 (1 \cdot e(H_{n+1})) = 12 \sigma_n. 
\]
%We also have $e(X_{12})=0$ since $X_{12} \rightarrow \bP^1$ is an abelian fibration and the Euler numbers of its singular fibers are zero. 
By these, we obtain 
$$e(X) = ((\gamma_m+6)\sigma_n + 21 \delta_n)+ (6 \sigma_n + 3 \delta_n) - 24 \sigma_n =   (\gamma_m-12) \sigma_n + 24 \delta_n.$$
\end{proof}

\begin{rem}
The above implies that the Euler number can be arbitrarily negative when $n$ is odd and can be arbitrarily positive when $n$ is even (except $n=2$). 
If $n=2$, then we compute $e(X)= 288$. We check that the formula (\ref{eq:Euler}) also holds when $n=1$. 
\end{rem}

\begin{rem}
We see that the Hodge to de Rham spectral sequence on $X$  degenerates at $E_1$ (cf. \cite[Section 4]{MR894379}, \cite[Corollary 11.24]{MR2393625}). 
It would be interesting whether our examples $X(m)$ satisfies the $\partial \bar{\partial}$-lemma and the hard Lefschetz property (cf.\cite{MR4085665},   \cite{MR3784517}). We hope to seek these  elsewhere. %It would require some work to check the $\partial \bar{\partial}$-lemma (cf. \cite[Corollary 1.6]{MR4085665}). 
\end{rem}

%\section{Higher dimensional case}
%Let $T \subset \bP^1 \times \bP^n$ be a general $(1,n+1)$-hypersurface. 
%We may write $T=(sF+tG =0) \subset \bP^1 \times \bP^n$ 
%for general $F, G \in H^0(\bP^n, \cO_{\bP^n}(n+1))$. 
%Then we see that the 2nd projection $T \rightarrow \bP^n$ is the blow-up along the subvariety $(F=G=0)$ and 
%$T \rightarrow \bP^1$ is a fibration whose general fibers are hypersurfaces in $\bP^n$ of degree $n+1$. 
%
%Let $Y_1:= \bP^2 \times T \subset \bP^2 \times \bP^1 \times \bP^n$. 

\section*{Acknowledgement}
The author is grateful to Kenji Hashimoto for useful discussions. 
He is also grateful to Nam-Hoon Lee for sending his manuscript and useful information. 
He thanks Jim Bryan, Robert Friedman and the anonymous referees for valuable comments. 
This work was partially supported by JSPS KAKENHI Grant Numbers JP17H06127, JP19K14509.

\bibliographystyle{amsalpha}
\bibliography{sanobibs-bddlogcy}

\def\cprime{$'$} \def\cprime{$'$}
\providecommand{\bysame}{\leavevmode\hbox to3em{\hrulefill}\thinspace}
\providecommand{\MR}{\relax\ifhmode\unskip\space\fi MR }
\providecommand{\MRhref}[2]{%
  \href{http://www.ams.org/mathscinet-getitem?mr=#1}{#2}
}
\providecommand{\href}[2]{#2}
\begin{thebibliography}{CLM19}

\bibitem[Cle77]{MR444662}
C.~H. Clemens, \emph{Degeneration of {K}\"{a}hler manifolds}, Duke Math. J.
  \textbf{44} (1977), no.~2, 215--290. \MR{444662}

\bibitem[CLM19]{Chan:2019vv}
Kwokwai Chan, Naichung~Conan Leung, and Ziming~Nikolas Ma, \emph{Geometry of
  the {M}aurer-{C}artan equation near degenerate {C}alabi-{Y}au varieties},
  https://arxiv.org/pdf/1902.11174.pdf (2019).

\bibitem[DI87]{MR894379}
Pierre Deligne and Luc Illusie, \emph{Rel{\`e}vements modulo {$p^2$} et
  d{\'e}composition du complexe de de {R}ham}, Invent. Math. \textbf{89}
  (1987), no.~2, 247--270. \MR{894379 (88j:14029)}

\bibitem[FFR19]{Felten:2019ve}
Simon Felten, Matej Filip, and Helge Ruddat, \emph{Smoothing toroidal crossing
  spaces}, https://arxiv.org/pdf/1908.11235.pdf (2019).

\bibitem[FLY12]{MR2891478}
Jixiang Fu, Jun Li, and Shing-Tung Yau, \emph{Balanced metrics on
  non-{K}\"{a}hler {C}alabi-{Y}au threefolds}, J. Differential Geom.
  \textbf{90} (2012), no.~1, 81--129. \MR{2891478}

\bibitem[FP20]{Felten:2020vm}
Simon Felten and Andrea Petracci, \emph{The logarithmic
  {B}ogomolov-{T}ian-{T}odorov theorem}, https://arxiv.org/pdf/2010.13656.pdf
  (2020).

\bibitem[Fri83]{MR707162}
Robert Friedman, \emph{Global smoothings of varieties with normal crossings},
  Ann. of Math. (2) \textbf{118} (1983), no.~1, 75--114. \MR{707162
  (85g:32029)}

\bibitem[Fri91]{MR1141199}
\bysame, \emph{On threefolds with trivial canonical bundle}, Complex geometry
  and {L}ie theory ({S}undance, {UT}, 1989), Proc. Sympos. Pure Math., vol.~53,
  Amer. Math. Soc., Providence, RI, 1991, pp.~103--134. \MR{1141199}

\bibitem[Fri19]{MR4085665}
\bysame, \emph{The {$\partial\overline\partial$}-lemma for general {C}lemens
  manifolds}, Pure Appl. Math. Q. \textbf{15} (2019), no.~4, 1001--1028.
  \MR{4085665}

\bibitem[GM93]{MR1240604}
Antonella Grassi and David~R. Morrison, \emph{Automorphisms and the
  {K}\"{a}hler cone of certain {C}alabi-{Y}au manifolds}, Duke Math. J.
  \textbf{71} (1993), no.~3, 831--838. \MR{1240604}

\bibitem[HS19]{Hashimoto:aa}
Kenji Hashimoto and Taro Sano, \emph{Examples of non-{K}{\"a}hler
  {C}alabi-{Y}au 3-folds with arbitrarily large $b_2$},
  https://arxiv.org/abs/1902.01027 (2019).

\bibitem[HSS98]{MR1672045}
Shinobu Hosono, Masa-Hiko Saito, and Jan Stienstra, \emph{On the mirror
  symmetry conjecture for {S}choen's {C}alabi-{Y}au {$3$}-folds}, Integrable
  systems and algebraic geometry ({K}obe/{K}yoto, 1997), World Sci. Publ.,
  River Edge, NJ, 1998, pp.~194--235. \MR{1672045}

\bibitem[KN94]{MR1296351}
Yujiro Kawamata and Yoshinori Namikawa, \emph{Logarithmic deformations of
  normal crossing varieties and smoothing of degenerate {C}alabi-{Y}au
  varieties}, Invent. Math. \textbf{118} (1994), no.~3, 395--409. \MR{1296351}

\bibitem[L\"09]{MR2569617}
Michael L\"{o}nne, \emph{Fundamental groups of projective discriminant
  complements}, Duke Math. J. \textbf{150} (2009), no.~2, 357--405.
  \MR{2569617}

\bibitem[Lee10]{MR2658406}
Nam-Hoon Lee, \emph{Calabi-{Y}au construction by smoothing normal crossing
  varieties}, Internat. J. Math. \textbf{21} (2010), no.~6, 701--725.
  \MR{2658406}

\bibitem[Lee21]{Lee:2021aa}
Nam-Hoon Lee, \emph{An example of non-{K}\"{a}hler {C}alabi-{Y}au fourfold},
  https://arxiv.org/pdf/2102.12656.pdf (2021).

\bibitem[LT96]{MR1410077}
P.~Lu and G.~Tian, \emph{The complex structures on connected sums of
  {$S^3\times S^3$}}, Manifolds and geometry ({P}isa, 1993), Sympos. Math.,
  XXXVI, Cambridge Univ. Press, Cambridge, 1996, pp.~284--293. \MR{1410077}

\bibitem[MOF]{MOF2020}
https://mathoverflow.net/questions/352878/calabi-yau-threefold-with-an-automorphism-of-infinite-order.

\bibitem[Nam91]{MR1093334}
Yoshinori Namikawa, \emph{On the birational structure of certain {C}alabi-{Y}au
  threefolds}, J. Math. Kyoto Univ. \textbf{31} (1991), no.~1, 151--164.
  \MR{1093334}

\bibitem[PS08]{MR2393625}
Chris A.~M. Peters and Joseph H.~M. Steenbrink, \emph{Mixed {H}odge
  structures}, Ergebnisse der Mathematik und ihrer Grenzgebiete. 3. Folge. A
  Series of Modern Surveys in Mathematics [Results in Mathematics and Related
  Areas. 3rd Series. A Series of Modern Surveys in Mathematics], vol.~52,
  Springer-Verlag, Berlin, 2008. \MR{2393625 (2009c:14018)}

\bibitem[QW18]{MR3784517}
Lizhen Qin and Botong Wang, \emph{A family of compact complex and symplectic
  {C}alabi-{Y}au manifolds that are non-{K}\"{a}hler}, Geom. Topol. \textbf{22}
  (2018), no.~4, 2115--2144. \MR{3784517}

\bibitem[Rei87]{MR909231}
Miles Reid, \emph{The moduli space of {$3$}-folds with {$K=0$} may nevertheless
  be irreducible}, Math. Ann. \textbf{278} (1987), no.~1-4, 329--334.
  \MR{909231}

\bibitem[Sch88]{MR923487}
Chad Schoen, \emph{On fiber products of rational elliptic surfaces with
  section}, Math. Z. \textbf{197} (1988), no.~2, 177--199. \MR{923487}

\bibitem[Tos15]{MR3372471}
Valentino Tosatti, \emph{Non-{K}\"{a}hler {C}alabi-{Y}au manifolds}, Analysis,
  complex geometry, and mathematical physics: in honor of {D}uong {H}. {P}hong,
  Contemp. Math., vol. 644, Amer. Math. Soc., Providence, RI, 2015,
  pp.~261--277. \MR{3372471}

\bibitem[Tot08]{MR2457523}
Burt Totaro, \emph{Hilbert's 14th problem over finite fields and a conjecture
  on the cone of curves}, Compos. Math. \textbf{144} (2008), no.~5, 1176--1198.
  \MR{2457523}

\end{thebibliography}

\end{document}